\documentclass[review]{elsarticle}
\usepackage{amsmath}

\usepackage{lipsum}
\usepackage{amsfonts}
\usepackage{graphicx}
\usepackage{epstopdf}
\usepackage{algorithmic}

\usepackage{graphicx}
\usepackage{epstopdf}
\usepackage{lineno,hyperref}
\usepackage{geometry}
\usepackage{epsfig}
\usepackage{lineno,hyperref}
\usepackage{bookmark}
\usepackage{multirow}
\usepackage{amsfonts}
\usepackage{multirow}
\usepackage{algorithmic}
\usepackage{algorithm}
\usepackage[figuresright]{rotating}
\usepackage{amscd}
\usepackage{mathrsfs}
\usepackage{booktabs,longtable}
\usepackage{indentfirst}
\usepackage{subfigure}
\usepackage{amstext}
\usepackage{amssymb}
\usepackage{fancyhdr}

\usepackage[nodots]{numcompress}

\usepackage[capitalize]{cleveref}

\usepackage{titlesec}
\usepackage{threeparttable}
\usepackage[title]{appendix}

\modulolinenumbers[5]
\biboptions{sort&compress}
\newtheorem{theorem}{Theorem}

\newtheorem{lemma}[theorem]{Lemma}
\newtheorem{remark}{Remark}

\newtheorem{condition}{Condition}

\newproof{proof}{\textbf{Proof}}

\journal{Journal of \LaTeX\ Templates}

\begin{document}
	
	\begin{frontmatter}
		
		\title{Optimal Sampling Algorithms for Block Matrix Multiplication\tnoteref{mytitlenote}}
		\tnotetext[mytitlenote]{The work is supported by the National Natural Science Foundation of China (No. 11671060) and the Natural Science Foundation Project of CQ CSTC (No. cstc2019jcyj-msxmX0267)}
		

		\author{Chengmei Niu}
		\ead{chengmeiniu@cqu.edu.cn.}
		
		\author  {Hanyu Li \corref{mycorrespondingauthor}}
			\cortext[mycorrespondingauthor]{Corresponding author}
		\ead{lihy.hy@gmail.com or hyli@cqu.edu.cn}

		\address{College of Mathematics and Statistics, Chongqing University, Chongqing 401331, P.R. China}
		
		\begin{abstract}
				In this paper, we investigate the randomized algorithms for block matrix multiplication from random sampling perspective.  Based on the A-optimal design criterion, the optimal sampling probabilities and sampling block sizes are obtained.  To improve the practicability of the  block sizes, two modified ones with less computation cost are provided. With respect to the second one, a two step algorithm is also devised. Moreover, the probability error bounds for the proposed algorithms are given.  Extensive numerical results show that our methods outperform the existing one in the literature.
		\end{abstract}
		
		
	\end{frontmatter}
	
	
	\section{Introduction}\label{sec.1}

As we know, matrix multiplication is a classical problem in numerical linear algebra. The algorithms of this problem are well-known and can be found in any book on matrix computations, see e.g. \cite{Golub2013}. However, in the age of big data, these famous algorithms have been encountered enormous challenges because of their computation cost. So, some scholars introduced the randomized ideas to matrix multiplication and proposed some randomized algorithms for this problem.

To the best of our knowledge, Cohen and Lewis \cite{Cohen1999} first applied the  randomized idea to approximate matrix multiplication. In 2006, motivated by a fast sampling algorithm for low-rank approximations given in \cite{Frieze2004}, Drineas et al.\cite{Drineas2006} proposed the now-famous randomized algorithm for matrix multiplication called the BasicMatrixMultiplication algorithm. It picks the outer products using the nonuniform sampling probabilities based on the norms of columns and rows of the involved matrices $M$ and $N$, respectively, that is, the following probabilities  
\begin{align}\label{1.1}
	p_i=\frac{\|M^{(i)}\|_2\|N_{(i)}\|_2}{\sum_{i = 1}^{n}\|M^{(i)}\|_2\|N_{(i)}\|_2}, 
	\ i=1,\cdots,n,
\end{align}
where $M^{(i)}$ denotes the $i$-th column of  $M \in \mathbb{R}^{m\times n}$,  $N_{(i)}$ stands for  the $i$-th row of $N \in \mathbb{R}^{n\times p}$, and $\|\cdot\|_2$ represents the Euclidean norm of a vector. The specific algorithm is given in Algorithm \ref{BMMA}.
\begin{algorithm}[htbp]
	\caption{BasicMatrixMultiplication Algorithm \cite{Drineas2006} }\label{BMMA}\small
	$\textbf{Input:}$  $M \in \mathbb{R}^{m\times n} $, $N \in \mathbb{R}^{n\times p}$, the number of sampling $c \in \mathbb{Z}^+$ such that $1\leq c \leq n$, and $\{p_i\}^{n}_{i=1}$ given as \eqref{1.1}. \\
	$\textbf{Output:}$ $C \in \mathbb{R}^{m\times c}$ and $D\in \mathbb{R}^{c\times p}.$
	\begin{enumerate}
		\item for $t=1$ to $c$,\\
		\begin{enumerate}
			\item  sample $i_t\in \{1,\cdots,n\}$ with $\rm{Pr}\mathit{{(i_t=s)=p_s}}$, $s=1,\cdots,n$,  	 	independently and with replacement.
			\item  set $C^{(t)}=\frac{M^{(i_t)}}{\sqrt{cp_{i_t}}}$, and $D_{(t)}=\frac{N_{(i_t)}}{\sqrt{cp_{i_t}}}$.
		\end{enumerate}
		\item return $C$ and $D$.
	\end{enumerate}
\end{algorithm}
Later, the
BasicMatrixMultiplication algorithm was extended to the block version by Wu \cite{Wu2018}. That is, a set of submatrices were sampled  by using the following sampling probabilities

\begin{align}\label{1.2}
	p_k= \frac{\|M^kN_k\|_F}{\sum_{k = 1}^{K}\|M^kN_k\|_F},\ k=1,\cdots,K,
\end{align}
where $M^k \in \mathbb{R}^{m\times n_k}$ represents the $k$-th block of $M=\begin{bmatrix}
	M^1 & M^2 & \cdots & M^K
\end{bmatrix}$, $N_k \in \mathbb{R}^{n_k\times p}$ symbolizes the $k$-th block of $N^T=\begin{bmatrix}
	N^T_1 & N^T_2& \cdots & N^T_K
\end{bmatrix}$, and $\|\cdot\|_F$ denotes the Frobenius norm of a matrix. In 2019,  
Chang et al. \cite{Chang2019}  proposed  another  block version of the BasicMatrixMultiplication algorithm  with the  following  sampling probabilities
\begin{align}\label{1.3}
	p_\mathbb{K}=\frac{\|\sum_{k \in \mathbb{K}}M^kN_k\|_F}{\sum_{\mathbb{K'}}\|\sum_{k \in \mathbb{K'}}M^kN_k\|_F}, 
\end{align}	
where $\mathbb{K} \subset \{\mathbb{K'}\}$ and $\mathbb{K'}$  denote the subsets of $\{1,2,3,\cdots,K\}$. Recently,
the following sampling probabilities	
\begin{align}\label{1.4}
	p_k= \frac{\|M^k\|_F\|N_k\|_F}{\sum_{k = 1}^{K}\|M^k\|_F\|N_k\|_F}, \ k=1,\cdots,K
\end{align}
were devised for the block matrix multiplication by Charalambides et al. \cite{Charalambides2020}. They are
easier to compute compared with \eqref{1.2} and \eqref{1.3}. In addition, there are some other generalizations of the BasicMatrixMultiplication algorithm \cite{Eriksson2011,Wu2020} and some randomized algorithms for matrix multiplication based on random  projection \cite{Cohen2015,Eftekhari2015,Srinivasa2020}. In particular,  a  block diagonal random projection method with different  block sizes  was developed in \cite{Srinivasa2020}.

In this paper, we consider the randomized algorithms for block matrix multiplication based on random sampling further using the technique of optimal subsampling proposed recently in the field of statistics\cite{Zhang2021,Zhu2015,Zuo2021}. Specifically, 
we  derive the optimal sampling probabilities and sampling block sizes  by the A-optimal design criterion\cite{Pukelsheim1993}. 
Moreover, unlike\cite{Wu2018,Chang2019,Charalambides2020},
we don't sample the blocks directly but sample the outer products on each block with the optimal sampling probabilities and sampling block sizes.

The remainder of this paper is organized as follows. The randomized algorithm for block matrix multiplication, optimal sampling probabilities and optimal sampling block sizes are presented in Section \ref{sec.2}. In Section \ref{sec.3}, we modify the block sizes to make them easier to compute and provide a two step algorithm. Furthermore, the probability error bounds of the corresponding algorithms are also given in Sections \ref{sec.2} and \ref{sec.3}, respectively. Extensive numerical experiments are shown in Section \ref{sec.4}. Finally, we make the concluding remarks of the whole paper.
	
	\section{Randomized Algorithm and Optimal Sampling Criterion}
	\label{sec.2}
	
	We first rewrite the product of the block matrices $M \in \mathbb{R}^{m\times n}$ and $N \in \mathbb{R}^{n\times p}$ appearing in Section \ref{sec.1}
as follows
\begin{eqnarray}\label{2.3}
	MN=\sum_{k = 1}^{K}{M^{k}N_{k}}=\sum_{k = 1}^{K}\sum_{i = 1}^{n_k}{M^{k(i)}N_{k(i)}},
\end{eqnarray}	
where $M^{k(i)}$ is viewed as the $i$-th column of the $k$-th block of $M$ and $N_{k(i)}$ is  the $i$-th row of the $k$-th block of $N$.	
Then,  Algorithm \ref{BMMA} is applied  to each block. Thus, we  have  $K$ estimates for  the $K$ blocks as follows
\begin{eqnarray}\label{3.2}
	C^kD_k=\sum_{t = 1}^{c_k}{C^{k(t)}D_{k(t)}}=\sum_{t = 1}^{c_k}\frac{M^{k(i_t)}N_{k(i_t)}}{c_kp_{k_{i_t}}},\ k=1,\cdots,K,
\end{eqnarray}
where $c_k$ represents the number of extracted outer products from the $k$-th block,
$C^{k(t)}=M^{k(i_t)}/\sqrt{c_kp_{k_{i_t}}}$ and $D_{k(t)}=N_{k(i_t)}/\sqrt{c_kp_{k_{i_t}}}$
with $p_{k_{i_t}}$ be the sampling probability satisfying $\sum_{i=1}^{n_k}p_{k_i}=1$. Note that these probabilities as well as the sampling block sizes $c_k$ with $k=1,\cdots,K$ need to be determined later in this section. Hence, they are undetermined at present. 
Therefore, the  final estimate is
\begin{eqnarray}\label{3.3}
	CD=\sum_{k = 1}^{K}C^kD_k=\sum_{k = 1}^{K}\sum_{t = 1}^{c_k}{C^{k(t)}D_{k(t)}}=\sum_{k = 1}^{K}\sum_{t = 1}^{c_k}\frac{M^{k(i_t)}N_{k(i_t)}}{c_kp_{k_{i_t}}}.
\end{eqnarray}
The specific algorithm is presented in Algorithm \ref{SABMM }.



\begin{algorithm}[htbp]
	\caption{Sampling Algorithm for Block Matrix Multiplication\ }\label{SABMM }\small
	$\textbf{Input:}$ $M \in \mathbb{R}^{m\times n} $ and $N \in \mathbb{R}^{n\times p}$ set as in Section \ref{sec.1},  $\{n_{k}\}^{K}_{k=1}$such that $\sum_{k = 1}^{K}n_k=n$, $\{c_{k}\}^{K}_{k=1}$ with $c_k \in \mathbb{Z}^+$ and $1\leq c_k \leq n_k$ such that $\sum_{k = 1}^{K}c_k=c$ for $c\in \mathbb{Z}^+$, and $\{p_{k_i}\}^{n_k}_{i=1}$ with $p_{k_i}\geq0$  such that $\sum_{i = 1}^{n_k}p_{k_i}=1$ for  $k=1,\cdots,K$. \\
	$\textbf{Output:}$ $C \in \mathbb{R}^{m\times c}$, $D\in \mathbb{R}^{c\times p}$, and $CD$.
	\begin{enumerate}
		\item for   $k\in{1,\cdots,K}$ do \\
		\qquad $[C^k,D_k]=$BasicMatrixMultiplication$(M^k,N_k,c_k,\{p_{k_i}\}^{n_k}_{i=1})$
		\item end
		\item $C=\begin{bmatrix}
			C^1 & C^2 & \cdots & C^K
		\end{bmatrix}$, $D^T=\begin{bmatrix}
			D^T_1 & D^T_2 & \cdots & D^T_K
		\end{bmatrix}$
		\item  $CD=\sum_{k = 1}^{K}C^kD_k $
		\item return  $C$, $D$, and $CD$		
	\end{enumerate}
\end{algorithm}

In the following,  we  discuss the asymptotic properties of the estimation obtained by Algorithm  \ref{SABMM }.
Based on these asymptotic properties and the A-optimal design criterion, we can construct the optimal sampling probabilities and sampling block sizes. Two conditions and a lemma are first listed as follows.

\begin{condition}\label{C.1}
	\begin{align}\label{3.4}
		\sum_{k=1}^{K}\frac{1}{ c^2_k}[(\sum_{i=1}^{n_k}\frac{|M^k_{(h,i)}||N_{k(i,f)}|}{p_{k_{i}}})^3]=o_p(1),
	\end{align}
	where $M^k_{(h,i)}$ with $h=1,\cdots,m$ and $i=1,\cdots,n_k$  represent the elements at the $(h,i)$-th position of the $k$-th block of $M$,   and  $N_{k(i,f)}$ with $f=1,\cdots,p$ and $i=1,\cdots,n_k$ denote  the $(i,f)$-th position  of the $k$-th block of $N$.
\end{condition}

\begin{condition}\label{C.2}
	\begin{align}\label{3.5}
		\sum_{k = 1}^{K}\sum_{i = 1}^{n_k}\frac{(M^k_{(h,i)})^2(N_{k(i,f)})^2}{\sqrt{c_k}p_{k_i}}={O}_p(1).
	\end{align}
	
\end{condition}
\begin{lemma}\label{lem3.1}
	The matrices $C$ and $D$ constructed by Algorithm \ref{SABMM } satisfy
	\begin{eqnarray}\label{3.6}
		\mathrm{E}{[(CD)_{(h,f)}]}=(MN)_{(h,f)}
	\end{eqnarray}
	and
	\begin{eqnarray}\label{3.7}
		\mathrm{Var}{[(CD)_{(h,f)}]}=\sum_{k=1}^{K}\sum_{i=1}^{n_k}\frac{(M^k_{(h,i)})^2(N_{k(i,f)})^2}{c_kp_{k_i}}-\sum_{k=1}^{K}\frac{((M^kN_k)_{(h,f)})^2}{c_k}.
	\end{eqnarray}
	
\end{lemma}
\begin{proof}
	The proof can be completed easily along the line of  the proof  of \cite[Lemma 3]{Drineas2006}.
\end{proof}

Now we present the asymptotic distribution of the estimation errors of matrix elements.

\begin{theorem}\label{thm3.1}
	Assume that Conditions \ \ref{C.1} and \  \ref{C.2} \ hold and let $c_s=\mathop{\min }\limits_{k=1,\cdots,K }c_k $.  Then the matrices $C$ and $D$ constructed by Algorithm \ref{SABMM } satisfy
	\begin{align}\label{thm2.1.1}
		\frac{(CD)_{(h,f)}-(MN)_{(h,f)}}{\sigma}\xrightarrow{L}N(0,1),\ for \ h=1,\cdots,m  \ and  \ f=1,\cdots,p,
	\end{align} where  $\xrightarrow{L}$ denotes  the  convergence in  distribution, and
	\begin{align*}
		\sigma^2=\sum_{k=1}^{K}\sum_{i=1}^{n_k}\frac{(M^k_{(h,i)})^2(N_{k(i,f)})^2}{c_kp_{k_i}}-\sum_{k=1}^{K}\frac{((M^kN_k)_{(h,f)})^2}{c_k}=O_p((\sqrt{c_s})^{-1}).
	\end{align*}	
\end{theorem}

\begin{proof}
	Note that
	\begin{align*}
		(CD)_{(h,f)}-(MN)_{(h,f)}&=\sum_{k=1}^{K}\sum_{t=1}^{c_k}(\frac{M^{k(i_t)}N_{k(i_t)}}{c_kp_{k_{i_t}}})_{(h,f)}-\sum_{k=1}^{K}\sum_{i=1}^{n_k}(M^{k(i)}N_{k(i)})_{(h,f)} \\
		&=\sum_{k=1}^{K}\sum_{t=1}^{c_k}[(\frac{M^{k(i_t)}N_{k(i_t)}}{c_kp_{k_{i_t}}})_{(h,f)}-\sum_{i=1}^{n_k}(\frac{M^{k(i)}N_{k(i)}}{c_k})_{(h,f)}].
	\end{align*}
	Let $\eta_{k(t)}=(\frac{M^{k(i_t)}N_{k(i_t)}}{c_kp_{k_{i_t}}})_{(h,f)}-\sum_{i=1}^{n_k}(\frac{M^{k(i)}N_{k(i)}}{c_k})_{(h,f)}$ with $k=1,\cdots,K$ and $t=1,\cdots,c_k$. Thus, based on Lemma 
	\ref{lem3.1},	it is easy to deduce that
	\begin{align}\label{thm2.1.2}
		\mathrm{E}[\eta_{k(t)}]=0
	\end{align}
	and 
	\begin{align}\label{thm2.1.3}	
		\mathrm{Var}[\eta_{k(t)}]=\sum_{i=1}^{n_k}\frac{(M^k_{(h,i)})^2(N_{k(i,f)})^2}{c^2_kp_{k_i}}-\frac{((M^kN_k)_{(h,f)})^2}{c^2_k}.
	\end{align}
	Moreover, considering that $\eta_{k(t)}$ are independent  for the given matrices $M$ and $N$, by the basic triangle inequality and the Cauchy-Schwarz inequality,  for any $\zeta>0$, we have	
	\begin{align}\label{thm2.1.4}	
		&\sum_{k=1}^{K}\sum_{t=1}^{c_k}\mathrm{E}{[\eta^2_{k(t)}I(|\eta_{k(t)}|>\zeta)|M,N]}
		\leq\sum_{k=1}^{K}\sum_{t=1}^{c_k}\frac{1}{\zeta}\mathrm{E}{[\eta^3_{k(t)}|M,N]} \notag \\
		&\leq\sum_{k=1}^{K}\sum_{t=1}^{c_k}\frac{1}{\zeta}\mathrm{E}{[|(\frac{M^k_{(h,i_t)}N_{k(i_t,f)}}{c_kp_{k_{i_t}}}-\sum_{i=1}^{n_k}\frac{M^k_{(h,i)}N_{k(i,f)}}{c_k})^3|]} \notag \\
		&\leq\sum_{k=1}^{K}\sum_{t=1}^{c_k}\frac{1}{\zeta c^3_k}[\sum_{i=1}^{n_k} \frac{|M^k_{(h,i)}|^3|N_{k(i,f)}|^3}{p^2_{k_{i}}}+(\sum_{i=1}^{n_k}|M^k_{(h,i)}||N_{k(i,f)}|)^3 \notag \\
		&\quad +3(\sum_{i=1}^{n_k}|M^k_{(h,i)}||N_{k(i,f)}|)(\sum_{i=1}^{n_k}|M^k_{(h,i)}||N_{k(i,f)}|)^2  +3\sum_{i=1}^{n_k} \frac{|M^k_{(h,i)}|^2|N_{k(i,f)}|^2}{p_{k_{i}}}(\sum_{i=1}^{n_k}|M^k_{(h,i)}||N_{k(i,f)}|)] \notag\\
		&\leq\sum_{k=1}^{K}\sum_{t=1}^{c_k}\frac{8}{\zeta c^3_k}[(\sum_{i=1}^{n_k}\frac{|M^k_{(h,i)}||N_{k(i,f)}|}{p_{k_{i}}})^3]
		=o_p(1),	
	\end{align}
	where the last equality  is from  Condition \ref{C.1}.  In addition, by \eqref{thm2.1.3}, we have  
	\begin{align}
		\sigma^2&=\sum_{k=1}^{K}\sum_{t=1}^{c_k}\mathrm{Var}[\eta_{k(t)}]\notag\\
		&=\sum_{k=1}^{K}\sum_{i=1}^{n_k}\frac{(M^k_{(h,i)})^2(N_{k(i,f)})^2}{c_kp_{k_i}}-\sum_{k=1}^{K}\frac{((M^kN_k)_{(h,f)})^2}{c_k} \notag\\
		&\leq\sum_{k=1}^{K}\sum_{i=1}^{n_k}\frac{(M^k_{(h,i)})^2(N_{k(i,f)})^2}{c_kp_{k_i}} \notag\\
		&\leq\frac{1}{\sqrt{ c_s}}\sum_{k=1}^{K}\sum_{i=1}^{n_k}\frac{(M^k_{(h,i)})^2(N_{k(i,f)})^2}{\sqrt{c_k}p_{k_i}}   \notag\\
		&=O_p((\sqrt{c_s})^{-1}), \label{thm2.1.5}
	\end{align}
	where the last inequality is derived by noting $c_s=\mathop{\min }\limits_{k=1,\cdots,K }c_k$ 
	and the last equality is based on Condition \ref{C.2}. 	
	Thus, 	combining \eqref{thm2.1.2}  
	\eqref{thm2.1.4} and 	\eqref{thm2.1.5},  by the  Lindeberg-Feller central limit theorem  \cite[Proposition 2.27]{Van1998}, we get \eqref{thm2.1.1}.

\end{proof}

\begin{remark}\label{rem2.1}
	When $c_s
	\rightarrow \infty$, the variance $\sigma^2$ can be negligible.
\end{remark}

Combining the A-optimal design criterion 
and the sum of  asymptotic variances of  elements, we can obtain the  optimal  sampling probabilities ${\{p_{k_i}\}}^{n_k}_{i=1}$ with $k=1,\cdots,K$ and the optimal sampling block sizes ${\{c_k\}}^K_{k=1}$ for Algorithm \ref{SABMM }. 

\begin{theorem}\label{thm3.2}
	For Algorithm \ref{SABMM }, 	the sum of the asymptotic variances,
	\begin{align*}
		\sum_{h = 1}^{m}\sum_{f = 1}^{p}\sigma^2
	\end{align*}
	attains its minimum 
	when
	\begin{eqnarray}\label{3.8}
		p_{k_i}=\frac{\|M^{k(i)}\|_2\|N_{k(i)}\|_2}{\sum_{i = 1}^{n_k}\|M^{k(i)}\|_2\|N_{k(i)}\|_2},\  for \ k=1,\cdots,K \ and \ i=1,\cdots,n_k,
	\end{eqnarray}
	and
	\begin{eqnarray}\label{3.9}
		c_{k}=c\frac{((\sum_{i = 1}^{n_k}\|M^{k(i)}\|_2\|N_{k(i)}\|_2)^2-\|M^kN_k\|^2_F)^{\frac{1}{2}}}{\sum_{k = 1}^{K}((\sum_{i = 1}^{n_k}\|M^{k(i)}\|_2\|N_{k(i)}\|_2)^2-\|M^kN_k\|^2_F)^{\frac{1}{2}}}, \ for \  k=1,\cdots,K.
	\end{eqnarray}	
	
\end{theorem}

\begin{proof}
	
	Considering 
	\begin{align*}
		(\sum_{i = 1}^{n_k}\|M^{k(i)}\|_2\|N_{k(i)}\|_2)^2-\|M^kN_k\|^2_F\geq 0 	
	\end{align*}
	and by the  Cauchy-Schwarz inequality,
	it is easy to get	
	\begin{align}\label{thm2.2.1}
		\sum_{h = 1}^{m}\sum_{f = 1}^{p}\sigma^2&=\sum_{k=1}^{K}\sum_{i=1}^{n_k}\frac{\|M^{k(i)}\|^2_2\|N_{k(i)}\|^2_2}{c_kp_{k_i}}-\sum_{k=1}^{K}\frac{\|M^kN_k\|^2_F}{c_k} \notag\\
		&=\sum_{k=1}^{K}\sum_{i=1}^{n_k}p_{k_i}\sum_{i=1}^{n_k}\frac{\|M^{k(i)}\|^2_2\|N_{k(i)}\|^2_2}{c_kp_{k_i}}-\sum_{k=1}^{K}\frac{\|M^kN_k\|^2_F}{c_k} \notag \\
		&\geq\sum_{k=1}^{K}\frac{1}{c_k}(\sum_{i=1}^{n_k}\|M^{k(i)}\|_2\|N_{k(i)}\|_2)^2-\sum_{k=1}^{K}\frac{\|M^kN_k\|^2_F}{c_k} \\
		&=\sum_{k=1}^{K}\frac{c_k}{c}\sum_{k=1}^{K}\frac{1}{c_k}[(\sum_{i=1}^{n_k}\|M^{k(i)}\|_2\|N_{k(i)}\|_2)^2-\|M^kN_k\|^2_F] \notag\\
		&\geq[\sum_{k=1}^{K}\frac{1}{\sqrt{c}}((\sum_{i=1}^{n_k}\|M^{k(i)}\|_2\|N_{k(i)}\|_2)^2-\|M^kN_k\|^2_F)^\frac{1}{2}]^2, \notag
	\end{align}
	where the equality  in \eqref{thm2.2.1}  holds  if and only if  $p_{k_i}$ are proportional to $\|M^{k(i)}\|_2\|N_{k(i)}\|_2$, i.e., $p_{k_i}=W_1\|M^{k(i)}\|_2\|N_{k(i)}\|_2$ for some constant $W_1\geq0$, and the  equality in the last inequality holds if and only if  $c_k=W_2[(\sum_{i=1}^{n_k}\|M^{k(i)}\|_2\|N_{k(i)}\|_2)^2-\|M^kN_k\|^2_F]^\frac{1}{2}$ for some $W_2\geq0$. Thus, considering  $\sum_{k=1}^{K}c_k=c$ and $\sum_{i=1}^{n_k}p_{k_i}=1$,   the desired results are derived.	
\end{proof}

\begin{remark}\label{rem2.2}
	It is not a complicated matter to find that
	\begin{align*}
		\sum_{h = 1}^{m}\sum_{f = 1}^{p}\sigma^2=\sum_{h = 1}^{m}\sum_{f = 1}^{p}\mathrm{Var}[(CD)_{(h,f)}]=\mathrm{E}{[\|MN-CD\|^2_F]},
	\end{align*}
	hence, the statistical criterion in Theorem \ref{thm3.2} for getting the optimal sampling probabilities and  sampling block sizes is equivalent to the optimization criterion used in \cite{Drineas2006}.
\end{remark}

\begin{remark}\label{rem2.3}
	Supposing $v_k\sum_{i=1}^{n_k}\|M^{k(i)}\|_2\|N_{k(i)}\|_2=\|M^kN_k\|_F$ for $0\leq v_k\leq1$, and combining  \eqref{3.8} and \eqref{3.9},  the sum of  asymptotic
	variances of  elements can be rewritten as
	\begin{align}\label{3.10}
		\sum_{h = 1}^{m}\sum_{f = 1}^{p}\sigma^2
		=\frac{1}{c}[\sum_{k=1}^{K}(1-v^2_k)^{\frac{1}{2}}\sum_{i=1}^{n_k}\|M^{k(i)}\|_2\|N_{k(i)}\|_2]^2.
	\end{align}
\end{remark}		

Next, we present the error bounds of the estimation obtained by Algorithm \ref{SABMM }. To make the analysis more general, we consider a set of sampling probabilities $\{p_{k_i}\}^{n_k}_{i=1}$ such that $p_{k_i}\geq\frac{\beta\|M^{k(i)}\|_2\|N_{k(i)}\|_2}{\sum_{i = 1}^{n_k}\|M^{k(i)}\|_2\|N_{k(i)}\|_2}$ with a positive constant $\beta\leq 1 $, which are named as the  nearly optimal probabilities. 

\begin{theorem}\label{thm3.3}
	Assume $v_k\sum_{i=1}^{n_k}\|M^{k(i)}\|_2\|N_{k(i)}\|_2=\|M^kN_k\|_F$ for $0\leq v_k\leq1$ and $\theta_1\leq1-v^2_k\leq\theta_2$ with $0\leq\theta_1\leq\theta_2\leq1$ and  $k=1,\cdots,K$,  and let  $\varphi=\frac{(\theta_2-\theta_1\beta+\theta_2\theta_1\beta)^\frac{1}{2}}{(\theta_2\theta_1)^{1/4}}$. Then, for Algorithm \ref{SABMM } applied  with $p_{k_i}\geq\frac{\beta\|M^{k(i)}\|_2\|N_{k(i)}\|_2}{\sum_{i = 1}^{n_k}\|M^{k(i)}\|_2\|N_{k(i)}\|_2}$ for $\beta\leq 1$  and 
	$c_{k}$ in \eqref{3.9}, the sum of the asymptotic variances satisfies 
	
	\begin{eqnarray}\label{3.11}
		\sum_{h = 1}^{m}\sum_{f = 1}^{p}\sigma^2\leq\frac{\varphi^2}{\beta c}\|M\|^2_F\|N\|^2_F,
	\end{eqnarray}
	and by setting $\delta\in(0,1)$ and $\eta=\varphi+(\frac{\theta_2}{\theta_1})^\frac{1}{2}\sqrt{(8/\beta)log(1/\delta)}$, 
	\begin{eqnarray}\label{3.12}
		{\|MN-{C}{D}\|^2_F }\leq{ \frac{\eta^2}{\beta c}\|M\|^2_F\|N\|^2_F}
	\end{eqnarray}
	holds with the probability at least $1-\delta$.
\end{theorem}

\begin{proof}
	Similar to the proof  of \cite[Theorem 1]{Drineas2006},  we  can derive the desired results. The specific proof is presented in Appendix \ref{app.A}.
\end{proof}
	
	\section{Modification of the Optimal  Criterion}
	\label{sec.3}
	
Note that calculating \eqref{3.9} requires to figure out the matrix multiplication $M^kN_k$. This cost may be prohibitive for massive data. In this section, we develop two low-cost alternatives,  $\hat{c_k}$ and $\widetilde{c_k}$, to replace the optimal sampling block size $c_k$ in \eqref{3.9}. Besides, a two step algorithm is also provided with respect to $\widetilde{c_k}$.

\subsection{Modification with Adjusting Variance} \label{sec.3.1}

The size $\hat{c}_k$ is derived from a small modification in the  proof of Theorem \ref{thm3.2}. That is, we first let 	
\begin{align}\label{4.1}
	\sum_{h = 1}^{m}\sum_{f = 1}^{p}\sigma^2\notag 
	&=\sum_{k=1}^{K}\sum_{i=1}^{n_k}\frac{\|M^{k(i)}\|^2_2\|N_{k(i)}\|^2_2}{\hat{c}_kp_{k_i}}-\sum_{k=1}^{K}\frac{\|M^kN_k\|^2_F}{\hat{c}_k} \notag \\
	&\leq\sum_{k=1}^{K}\sum_{i=1}^{n_k}\frac{\|M^{k(i)}\|^2_2\|N_{k(i)}\|^2_2}{\hat{c}_kp_{k_i}},
\end{align}
and then find two sets  $\{\hat{c}_k\}^K_{k=1}$ and $\{p_{k_i}\}^{n_k}_{i=1}$  to make the above upper bound  achieve minimum. Similar to the proof of Theorem \ref{thm3.2}, we have 
\begin{align}\label{4.2}
	\hat{c}_k=c\frac{\sum_{i = 1}^{n_k}\|M^{k(i)}\|_2\|N_{k(i)}\|_2}{\sum_{k = 1}^{K}\sum_{i = 1}^{n_k}\|M^{k(i)}\|_2\|N_{k(i)}\|_2}
\end{align}
and  $p_{k_i}$  as in \eqref{3.8}.	
Obviously,  $\hat{c}_k$ is much easier to compute compared with  \eqref{3.9}.

Below we provide the asymptotic distribution  of the estimation errors of matrix elements and  probability error bound of $\hat{C}\hat{D}$ constructed by  putting  \eqref{3.8} and \eqref{4.2} into Algorithm \ref{SABMM }. 	The following conditions are first listed.

\begin{condition}\label{C.3}
	\begin{align*}
		\frac{1}{c^2}(\sum_{k = 1}^{K}\sum_{i = 1}^{n_k}\|M^{k(i)}\|_2\|N_{k(i)}\|_2)^3=o_p(1).
	\end{align*}
\end{condition}

\begin{condition}\label{C.4}
	\begin{align*}
		\frac{1}{\sqrt{c}}(\sum_{k = 1}^{K}\sum_{i = 1}^{n_k}\|M^{k(i)}\|_2\|N_{k(i)}\|_2)^2=O_p(1).
	\end{align*}
\end{condition} 		

\begin{theorem}\label{thm4.1}
	Assume  $\mu_1L\leq|M_{(h,f)}|\leq\mu_2L$ and $\mu_1L\leq|N_{(h,f)}|\leq\mu_2L$  for $\mu_1\geq0$, $\mu_2\geq0$ and $L\geq0$, and set $c_s=\mathop{\min }\limits_{k=1,\cdots,K }\hat{c_k}$.	
	If  Conditions \ref{C.3} and \ref{C.4} hold, then the matrices
	$\hat{C}$ and $\hat{D}$ constructed by Algorithm \ref{SABMM } with $p_{k_i}$ being as in \eqref{3.8} and $c_k=\hat{c}_k$ satisfy 
	\begin{align*}
		\frac{(\hat{C}\hat{D})_{(h,f)}-(MN)_{(h,f)}}{\sigma}\xrightarrow{L}N(0,1), for \ h=1,\cdots,m \ and \ f=1,\cdots,p,
	\end{align*}
	where  $\xrightarrow{L}$ denotes  the  convergence in  distribution, and
	\begin{align*}
		\sigma^2=\sum_{k = 1}^{K}\sum_{i = 1}^{n_k}\|M^{k(i)}\|_2\|N_{k(i)}\|_2[\sum_{k=1}^{K}\frac{1}{c}
		(\sum_{i=1}^{n_k}\frac{(M^k_{(h,i)})^2(N_{k(i,f)})^2}{\|M^{k(i)}\|_2\|N_{k(i)}\|_2}-\frac{((M^kN_k)_{(h,f)})^2}{\sum_{i = 1}^{n_k}\|M^{k(i)}\|_2\|N_{k(i)}\|_2})]
		=O_p((\sqrt{c_s})^{-1}).
	\end{align*}
\end{theorem}

\begin{proof}
	The proof can be completed along the line of  the proof of Theorem \ref{thm3.1}. 
\end{proof}

\begin{theorem}\label{thm4.2}
	Let  $\hat{\varphi}=(1-\beta(1-\theta_2))^\frac{1}{2}$ with $\theta_2$ given as in Theorem \ref{thm3.3}.
	Then, for  Algorithm \ref{SABMM } applied with $p_{k_i}\geq\frac{\beta\|M^{k(i)}\|_2\|N_{k(i)}\|_2}{\sum_{i = 1}^{n_k}\|M^{k(i)}\|_2\|N_{k(i)}\|_2}$ for  $\beta\leq 1$ and $c_k  =  \hat{c}_k$, the sum of the asymptotic variances satisfies
	\begin{eqnarray}\label{4.3}
		\sum_{h = 1}^{m}\sum_{f = 1}^{p}\sigma^2\leq\frac{\hat{\varphi}^2}{\beta c}\|M\|^2_F\|N\|^2_F,
	\end{eqnarray}
	and by setting $\delta\in(0,1)$  and $\eta=\hat{\varphi}+\sqrt{(8/\beta)log(1/\delta)}$, 
	\begin{eqnarray}\label{4.4}
		{\|MN-\hat{C}\hat{D}\|^2_F }\leq{ \frac{\eta^2}{\beta c}\|M\|^2_F\|N\|^2_F}
	\end{eqnarray}
	holds with the probability at least $1-\delta$.
\end{theorem}

\begin{proof}
	The proof can be completed along the line of  the proof of Theorem \ref{thm3.3}.
\end{proof}

\begin{remark}\label{rem4.4}
	Let $\theta_2=\theta_1<1$ and $\beta=1$  in Theorem \ref{thm4.2} we have $\eta=(\theta_2)^{1/2}+\sqrt{(8/\beta)log(1/\delta)}$. In this case,  the corresponding probability error bound is the same as the one in Theorem \ref{thm3.3}.	
\end{remark}

\subsection{ Modification with the BasicMatrixMultiplication Algorithm}\label{sec.3.2}

We can  use $C^{0k}D_{0k}$ constructed by Algorithm \ref{BMMA} with  the same sampling size $[c0/K]$ and a set of sampling probabilities  ${\{p_{0k_i}\}}^{n_k}_{i=1}$ to approximate $M^kN_k$, where  $c0$ denotes  the total sample size,
and  ${p_{0k_i}}$  with $i=1,\cdots,n_k$ are allowed to be uniform probabilities or nonuniform probabilities. Considering that   $(\sum_{i = 1}^{n_k}\|M^{k(i)}\|_2\|N_{k(i)}\|_2)^2-\|C^{0k}D_{0k}\|^2_F\geq0$ may not hold,  we propose $\widetilde{c}_k$ as follows
\begin{align}\label{4.7}
	\widetilde{c}_k=c\frac{(|(\sum_{i = 1}^{n_k}\|M^{k(i)}\|_2\|N_{k(i)}\|_2)^2-\|C^{0k}D_{0k}\|^2_F|)^{\frac{1}{2}}}{\sum_{k = 1}^{K}(|(\sum_{i = 1}^{n_k}\|M^{k(i)}\|_2\|N_{k(i)}\|_2)^2-\|C^{0k}D_{0k}\|^2_F|)^{\frac{1}{2}}}.
\end{align}
Based on the above idea, we  devise a two step algorithm summarized in Algorithm \ref{TSTABMM}.

\begin{algorithm}[htbp]
	\caption{Two Step Algorithm for Block Matrix Multiplication  }\label{TSTABMM}\small
	$\textbf{Input:}$ $M \in \mathbb{R}^{m\times n} $ and $N \in \mathbb{R}^{n\times p}$ set as in Section \ref{sec.1},  $\{n_{k}\}^{K}_{k=1}$such that $\sum_{k = 1}^{K}n_k=n$, $c\in \mathbb{Z}^+$, $c_0 \in \mathbb{Z}^+$,  and $\{p_{0k_i}\}^{n_k}_{i=1}$ with $p_{0k_i}\geq0$  such that $\sum_{i = 1}^{n_k}p_{0k_i}=1$, $k=1,\cdots,K$.  \\
	$\textbf{Output:}$ $\widetilde{C} \in \mathbb{R}^{m\times c}$, $\widetilde{D}\in \mathbb{R}^{c\times p}$, and	$\widetilde{C}\widetilde{D}$.\\		
	$\textbf{Step1:}$
	\begin{enumerate}
		\item for   $k\in{1,\cdots,K}$ do \\
		\quad  update $[C^{0k},D_{0k}]=$BasicMatrixMultiplication$(M^k,N_k,[c0/K],\{p_{0k_i}\}^{n_k}_{i=1})$\\
		\quad update $p_{k_i}=\frac{\|M^{k(i)}\|_2\|N_{k(i)}\|_2}{\sum_{i = 1}^{n_k}\|M^{k(i)}\|_2\|N_{k(i)}\|_2}, i=1,\cdots,n_k$\\
		\item end
		\item  replace $\|M^kN_k\|^2_F$ in \eqref{3.9} by $C^{0k}D_{0k}$, i.e., 
		$\widetilde{c}_k=c\frac{(|(\sum_{i = 1}^{n_k}\|M^{k(i)}\|_2\|N_{k(i)}\|_2)^2-\|C^{0k}D_{0k}\|^2_F|)^{\frac{1}{2}}}{\sum_{k = 1}^{K}(|(\sum_{i = 1}^{n_k}\|M^{k(i)}\|_2\|N_{k(i)}\|_2)^2-\|C^{0k}D_{0k}\|^2_F|)^{\frac{1}{2}}}$\\
		\item return
		$\widetilde{c}_k$ for  $k=1,\cdots,K$,  and $p_{k_i}$ for $k=1,\cdots,K$ and $i=1,\cdots,n_k$
	\end{enumerate}
	$\textbf{Step2:}$
	\begin{enumerate}
		\item for   $k\in{1,\cdots,K}$ do \\
		\qquad
		$[\widetilde{C}^k,\widetilde{D}_k]=$BasicMatrixMultiplication$(M^k,N_k,\widetilde{c_k},\{p_{k_i}\}^{n_k}_{i=1})$
		\item end
		
		\item $\widetilde{C}=\begin{bmatrix}
			\widetilde{C}^1 & \widetilde{C}^2 & \cdots & \widetilde{C}^K
		\end{bmatrix}$, $\widetilde{D}^T=\begin{bmatrix}
			\widetilde{D}^T_1 & \widetilde{D}^T_2 & \cdots & \widetilde{D}^T_K
		\end{bmatrix}$
		\item
		$\widetilde{C}\widetilde{D}=\sum_{k = 1}^{K}\widetilde{C}^k\widetilde{D}_k $
		\item return
		$\widetilde{C}$, $\widetilde{D}$, and	$\widetilde{C}\widetilde{D}$		
	\end{enumerate}
\end{algorithm}

Now, we provide the asymptotic distribution  of the estimation errors of matrix elements and probability error bound of $\widetilde{C}\widetilde{D}$ obtained by Algorithm \ref{TSTABMM}.

\begin{theorem}\label{thm4.3}
	Suppose  $\hat{v}_k\sum_{i=1}^{n_k}\|M^{k(i)}\|_2\|N_{k(i)}\|_2=\|C^{0k}D_{0k}\|_F$, for  $\hat{\theta}_1\leq|1-\hat{v}^2_k|\leq\hat{\theta}_2$ with $k=1,\cdots,K$, 
	set $\mu_1$, $\mu_2$ and $L$ as in Theorem \ref{thm4.1}, and let $c_s=\mathop{\min }\limits_{k=1,\cdots,K }\widetilde{c_k}$.
	If   Conditions  \ref{C.3} and \ref{C.4} hold, then the matrices $\widetilde{C}$ and $\widetilde{D}$ constructed by  
	Algorithm \ref{TSTABMM}, satisfy
	\begin{align*}
		\frac{(\widetilde{C}\widetilde{D})_{(h,f)}-(MN)_{(h,f)}}{\sigma}\xrightarrow{L}N(0,1), for \ h=1,\cdots,m \  and \  f=1,\cdots,p,
	\end{align*} 
	where  $\xrightarrow{L}$ denotes  the  convergence in  distribution, and
	\begin{align*}
		\sigma^2&=\sum_{k = 1}^{K}|1-\hat{v}^2_k|^\frac{1}{2}\sum_{i = 1}^{n_k}\|M^{k(i)}\|_2\|N_{k(i)}\|_2[\sum_{k=1}^{K}\frac{1}{c|1-\hat{v}^2_k|^\frac{1}{2}}(\sum_{i=1}^{n_k}\frac{(M^k_{(h,i)})^2(N_{k(i,f)})^2}{\|M^{k(i)}\|_2\|N_{k(i)}\|_2}-\frac{((M^kN_k)_{(h,f)})^2}{\sum_{i = 1}^{n_k}\|M^{k(i)}\|_2\|N_{k(i)}\|_2})]\\
		&	=O_p((\sqrt{c_s})^{-1}).
	\end{align*}	
\end{theorem}

\begin{proof}
	The proof can be completed along the line of  the proof of Theorem \ref{thm3.1}.
\end{proof}	

\begin{theorem}\label{thm4.4}
	Assume the same setting as in Theorem \ref{thm4.3}  and let  $\widetilde{\varphi}=\frac{(\hat{\theta}_2-\hat{\theta}_1\beta+\theta_2\hat{\theta}_1\beta)^\frac{1}{2}}{(\hat{\theta}_2\hat{\theta}_1)^1/4}$ with $\theta_2$ given as in Theorem \ref{thm3.3}. Then, for Algorithm \ref{TSTABMM} applied with $p_{k_i}\geq\frac{\beta\|M^{k(i)}\|_2\|N_{k(i)}\|_2}{\sum_{i = 1}^{n_k}\|M^{k(i)}\|_2\|N_{k(i)}\|_2}$  for $\beta\leq1$, the sum of the asymptotic variances satisfies
	\begin{eqnarray}\label{4.8}
		\sum_{h = 1}^{m}\sum_{f = 1}^{p}\sigma^2\leq\frac{\widetilde{\varphi}^2}{\beta c}\|M\|^2_F\|N\|^2_F,
	\end{eqnarray}
	and by setting $\delta\in(0,1)$ and $\eta=\widetilde{\varphi}+(\frac{\hat{\theta}_2}{\hat{\theta}_1})^\frac{1}{2}\sqrt{(8/\beta)log(1/\delta)}$, then
	\begin{eqnarray}\label{4.9}
		{\|MN-\widetilde{C}\widetilde{D}\|^2_F }\leq{ \frac{\eta^2}{\beta c}\|M\|^2_F\|N\|^2_F}
	\end{eqnarray}
	holds with the probability at least $1-\delta$.
\end{theorem}

\begin{proof}
	The proof can be completed  along the line of  the proof of the proof of Theorem \ref{thm3.3}.
\end{proof}

\begin{remark}\label{rem4.6}
	When  $\hat{\theta}_2>1$, $\hat{\theta}_1\leq1$ and $1-\theta_2=o_p(1)$,  the bounds in Theorem \ref{thm4.4} is a little weaker than the one in  Theorem \ref{thm4.2}.	
\end{remark}

\section{Numerical Experiments}
\label{sec.4}

Without loss of generality, we set the sizes of the blocks of the involved block matrices $M$ and $N$ to be the same, namely $n_k=n/K$ for $k=1,\cdots,K$. To construct the following matrices $M$ and $N$, we let $m=26$, $p=28$, $n=5\times10^5$, $\Sigma_1=(1 \times 0.7^{|i-j|})$ with  $1\leq i,j \leq m$ and $\Sigma_2=(2 \times 0.7^{|i-j|})$  with  $1\leq i,j \leq p$.

Case I:  The  $i$-th column of $M$ with $1\leq i \leq m$, $M^{(i)}$,  is generated from a multivariate normal distribution, that is, $M^{(i)} \sim N(\rm{0}, \Sigma_1)$. Similarly, set $N_{(j)} \sim N(\rm{0}, \Sigma_2)$.

Case II: The  $i$-th column of $M$ with $1\leq i \leq m$, $M^{(i)}$, is generated  from a multivariate $t$ distribution with 1 degree of freedom, that is, $M^{(i)} \sim t_1(\rm{1}, \Sigma_1)$. Similarly, set $N_{(j)} \sim t_1(\rm{1}, \Sigma_2)$.

For the above matrices, by setting suitable values of $K$, $c0$ and $c$,  we do  five specific experiments summarized  in Table \ref{Tab51} and report the numerical results on accuracy and CPU time  in Figures 1-10.  Note that all the numerical results are based on 100 replications, and in these figures, SSM denotes the method from \cite{Charalambides2020}, whose sampling probabilities are  given in \eqref{1.4}, and
other notations for the methods are introduced in Table \ref{Tab52}. 

\begin{table}[H]
	\scriptsize		
	\centering

	
	\begin{threeparttable}
		
		\begin{tabular}{cccccccccc}
			
			\hline\noalign{\smallskip} 
			
			Number&	Comparison&  $c$ &  $K$  &  $c0$  & Results  \\
			
			\noalign{\smallskip}\hline\noalign{\smallskip}
			1&	  Algorithm \ref{SABMM } and SSM &  $50000$ to $5\times10^5$ & $10$ &$5000$	&  Figures 1-2			
			\\		
			2& Algorithm \ref{SABMM } and SSM & $50000$  & $10$ to $500$  &$5000$	 &  Figures 3-4	
			\\	
			3&	Algorithms \ref{SABMM } and  \ref{TSTABMM} & $50000$ to $5\times10^5$ & $10$ &$5000$	&  Figures 5-6
			\\	
			4&	Algorithms \ref{SABMM } and   \ref{TSTABMM} &  $50000$  & $10$ to $500$ &$5000$ &  Figures 7-8
			\\
			5&	Algorithms \ref{SABMM } and   \ref{TSTABMM} &  $50000$ & $10$ &$500$ to $50000$&  Figures 9-10
			\\	
			\noalign{\smallskip}\hline	
			
		\end{tabular}
		
	\end{threeparttable}

	\caption{Description of five experiments}
	\label{Tab51}
	
\end{table}

\begin{table}[H]
	\scriptsize		
	\centering
	
	
	
	\begin{threeparttable}

		\begin{tabular}{cccccccccc}

			\hline\noalign{\smallskip} 
			
			Method &$p_{k_i}$& $c_k$ &    $p_{0k_i}$ \\
			
			\noalign{\smallskip}\hline\noalign{\smallskip}
			ONU (from Algorithm \ref{TSTABMM})& 	 \eqref{3.8}	& 	 	$\widetilde{c_k}$ & 	 	$\frac{1}{n_k}$ 			
			\\		
			ONMCNR (from Algorithm \ref{TSTABMM})& 		 \eqref{3.8}& 		$\widetilde{c_k}$ & 		$\frac{\|M^{k(i)}\|_2\|N_{k(i)}\|_2}{\sum_{i = 1}^{n_k}\|M^{k(i)}\|_2\|N_{k(i)}\|_2}$ 	
			\\	
			OPL (from Algorithm \ref{SABMM })  	&	 \eqref{3.8}& 	  	 \eqref{3.9} & 	 Null 	
			\\	
			ONC (from Algorithm \ref{SABMM })& 	  	 \eqref{3.8}& 	 	$\hat{c_k}$ 	 & 	Null
			\\
			UU (from Algorithm \ref{SABMM })& 		$\frac{1}{n_k}$ 	 & 	$\frac{c}{K}$ & 	  Null		
			\\	
			\noalign{\smallskip}\hline	
			
		\end{tabular}

		
	\end{threeparttable}
\caption{Explanation  of sampling methods with different probabilities and block sizes}
	\label{Tab52}

\end{table}


In the first two experiments,   we compare Algorithm \ref{SABMM } and  SSM for  different $c$ and $K$, respectively. The corresponding numerical results  are shown in  Figures 1-4. From these figures, we can find that, for Case II, OPL and ONC outperform SSM in accuracy for different $c$ or different $K$, however, they need more computing time. While, the improvement in accuracy is more than the increasement in computing time.  For Case I, the four methods have similar performances in accuracy, and  OPL and ONC are a little expensive. These findings are consistent with the theoretical results of these methods. Furthermore, it is interesting to find that UU may be superior to SSM in accuracy and CPU time for Case II.

\begin{figure}[H]
	\centering
	\includegraphics[width=14cm]{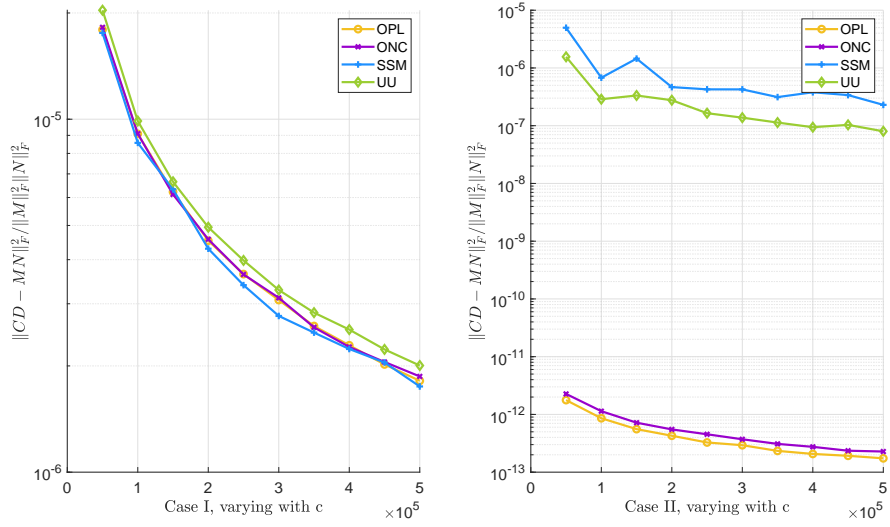} 
	\caption{Comparison of relative errors for Algorithm 2 and SSM varying  with $c$} 
	\label{fig51}
\end{figure}

\begin{figure}[H]
	\centering
	\includegraphics[width=14cm]{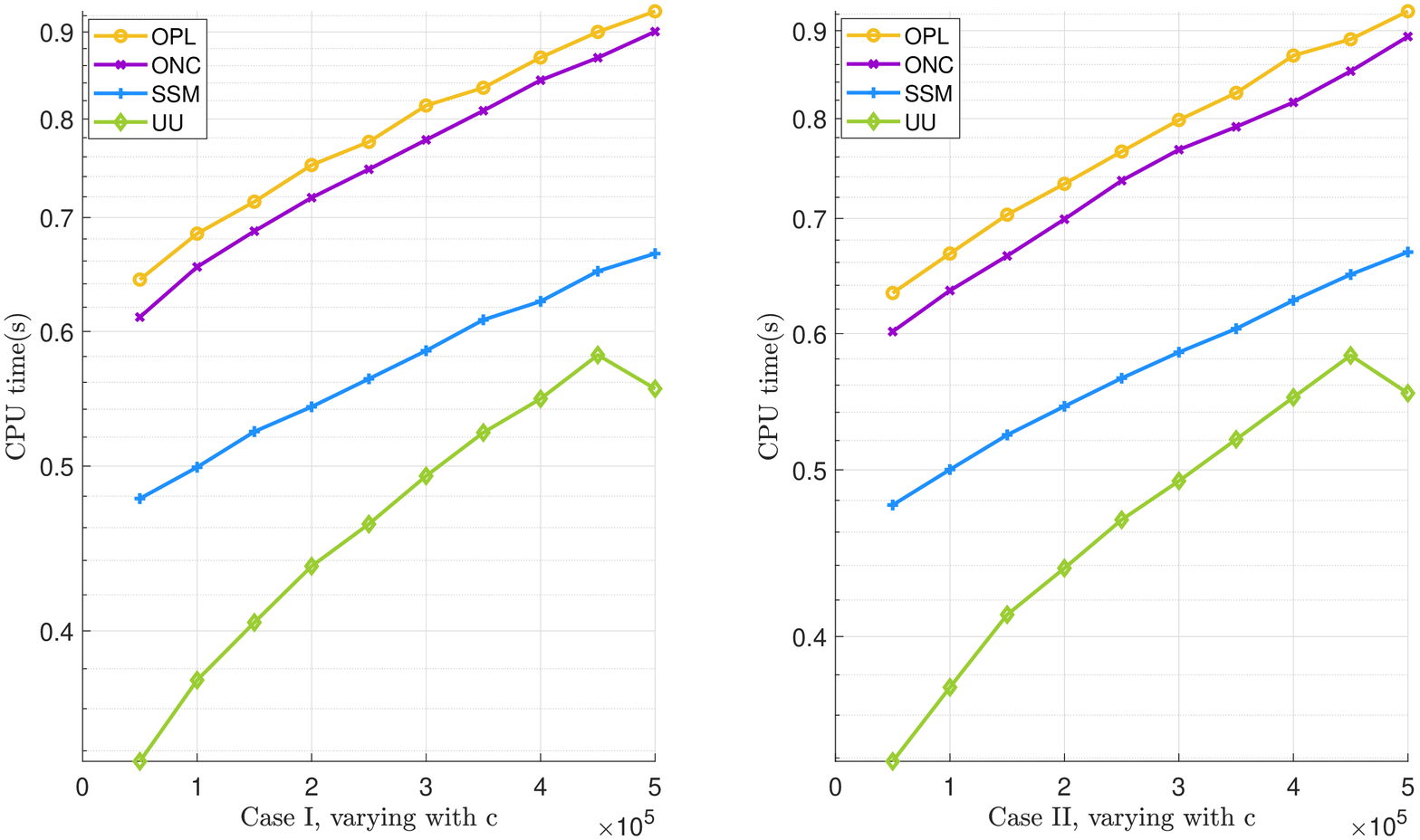} 
	\caption{Comparison of CPU time for Algorithm 2 and SSM varying  with $c$} 
	\label{fig52} 
\end{figure}

\begin{figure}[H]
	\centering
	\includegraphics[width=14cm]{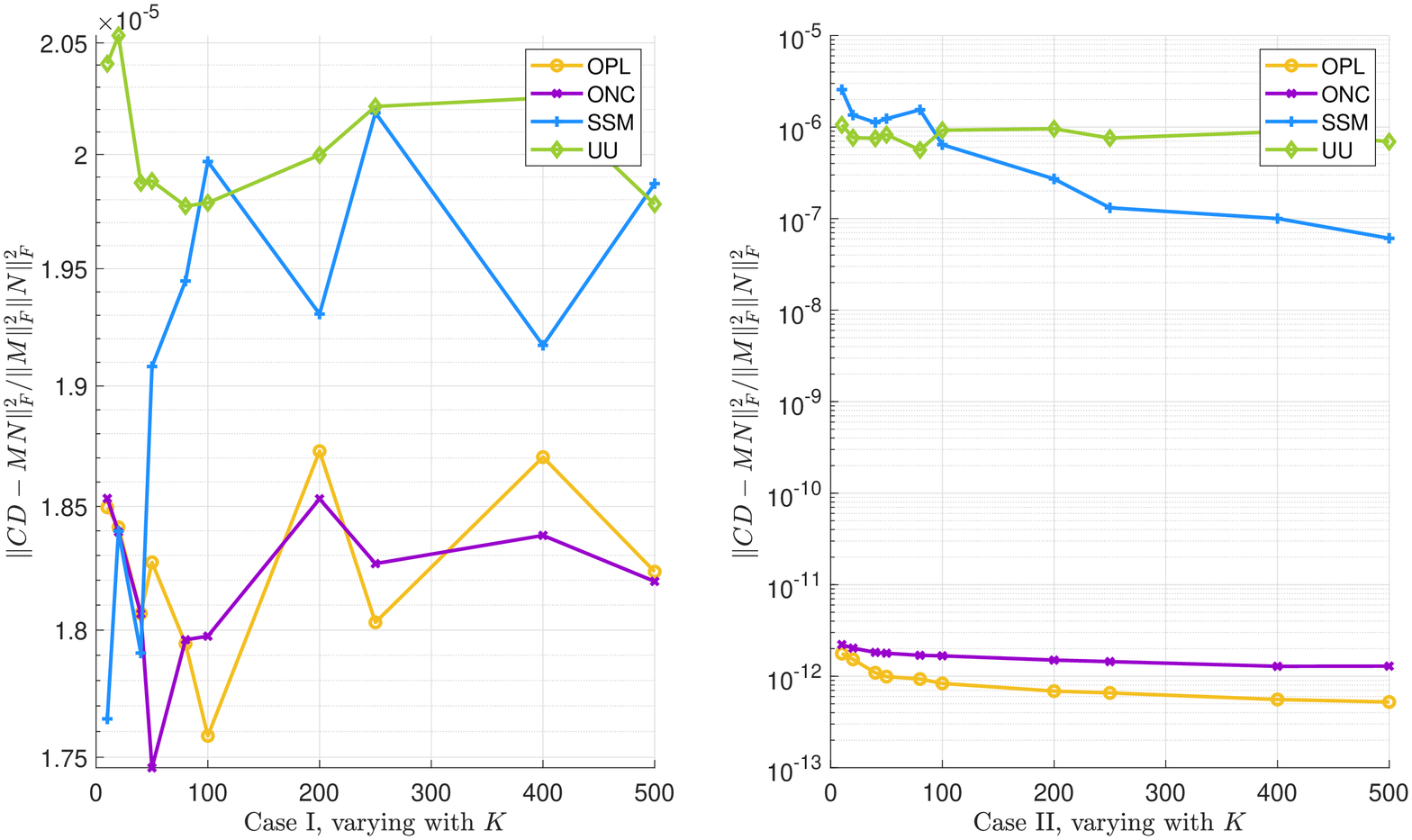} 
	\caption{Comparison of relative errors for Algorithm 2 and SSM varying  with $K$} 
	\label{fig53}
\end{figure}

\begin{figure}[H]
	\centering
	\includegraphics[width=14cm]{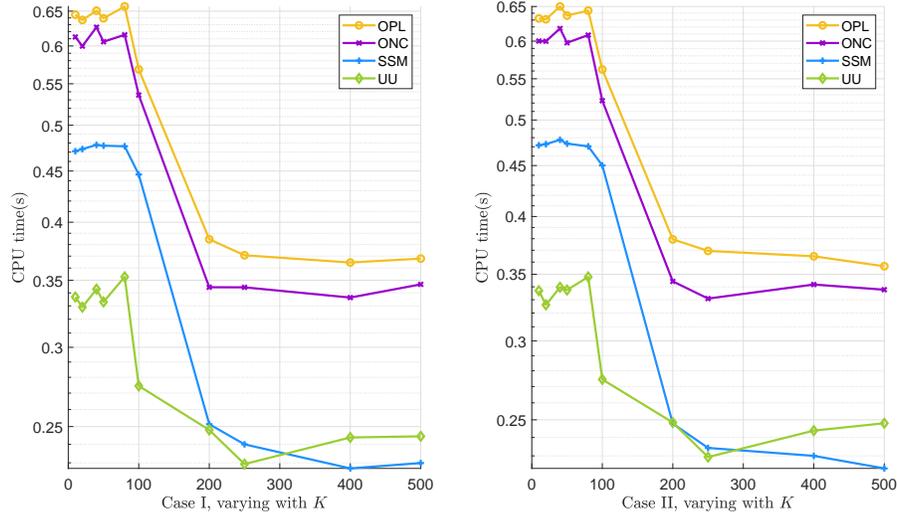} 
	\caption{Comparison of CPU time for Algorithm 2 and SSM varying  with $K$} 
	\label{fig54}
\end{figure}

The third and fourth experiments are utilized to compare Algorithms \ref{SABMM } and \ref{TSTABMM} for different $c$ and  different $K$, respectively.  Based on the numerical results  presented in Figures 5-8,  we get that, for Case II,  OPL always performs best in accuracy and needs most CPU time in most of cases. For different $c$, ONMCNR has the similar performance in accuracy to OPL, however, for large $K$, i.e., small $c0/K$, it has the worst accuracy. It is a little strange and we have no  reasonable explanation at present. In addition, ONC always performs quite well. It needs  least CPU time but has the similar accuracy to OPL.  For Case I, the four methods perform similarly in accuracy and CPU time.

\begin{figure}[H]
	\centering
	\includegraphics[width=14cm]{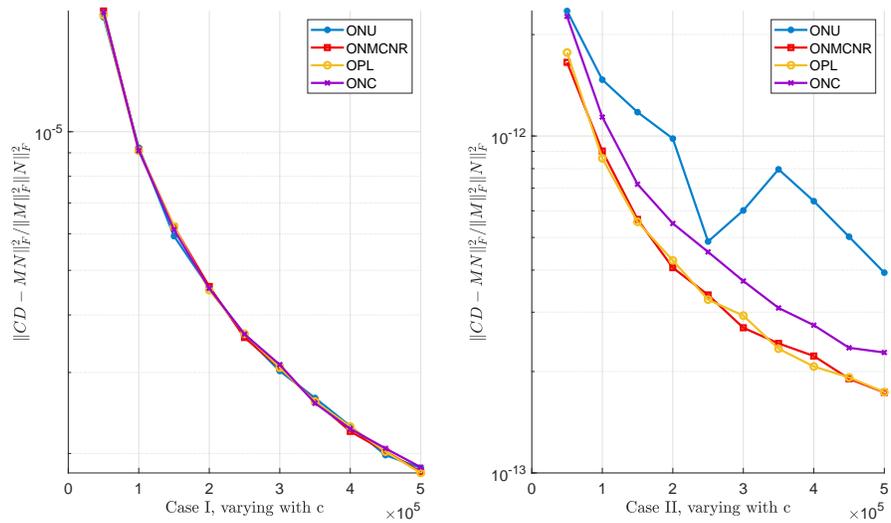} 
	\caption{Comparison of relative errors for Algorithms 2 and  3 varying  with $c$} 
	\label{fig55}
\end{figure}


\begin{figure}[H]
	\centering
	\includegraphics[width=14cm]{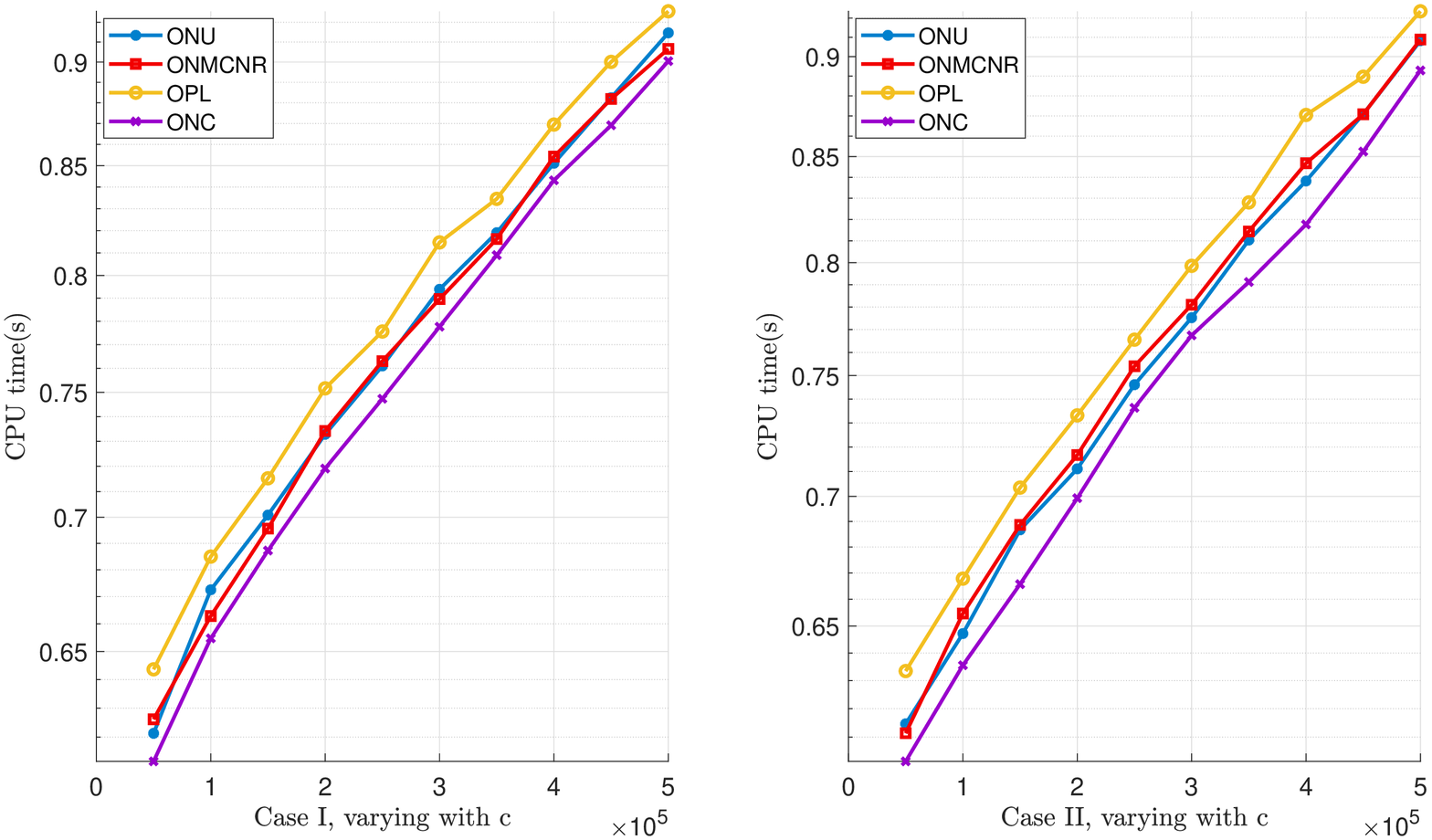} 
	\caption{Comparison of CPU time for Algorithms 2 and  3 varying  with $c$} 
	\label{fig56}
\end{figure}

\begin{figure}[H]
	\centering
	\includegraphics[width=14cm]{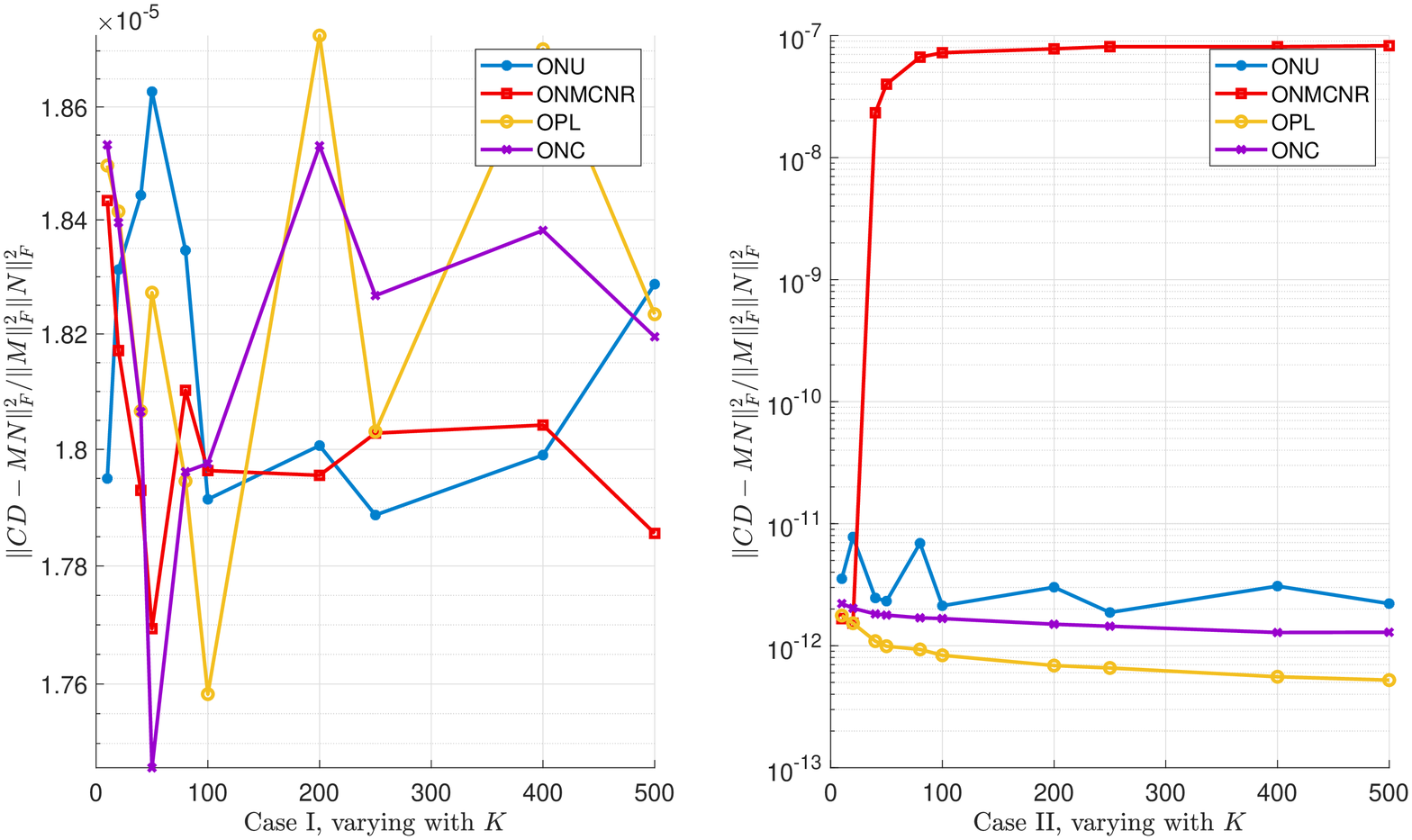} 
	\caption{Comparison of relative errors for Algorithms 2 and  3   varying  with $K$} 
	\label{fig57}
\end{figure}

\begin{figure}[H]
	\centering
	\includegraphics[width=14cm]{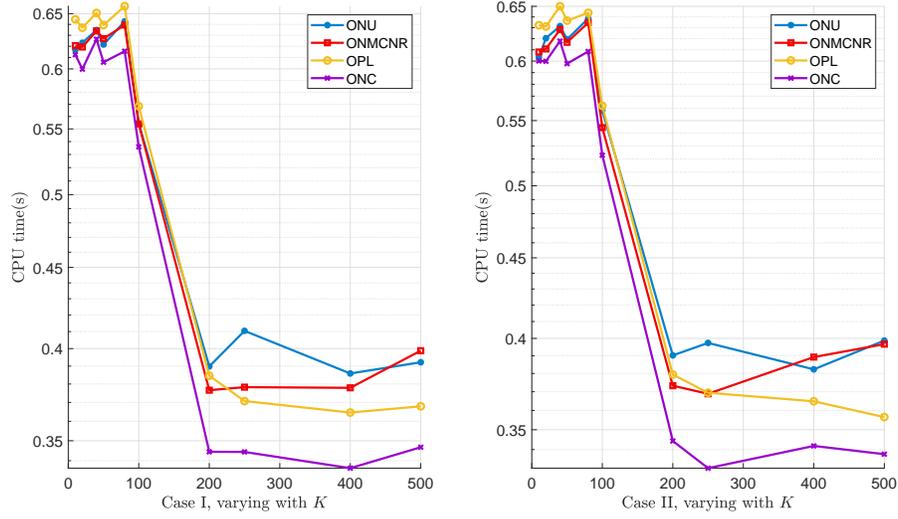} 
	\caption{Comparison of CPU time for Algorithms 2 and  3   varying  with $K$} 
	\label{fig58}
\end{figure}

In the  last experiment, we compare Algorithms \ref{SABMM } and \ref{TSTABMM} for different $c0$. The corresponding numerical results are shown in  Figures  9 and 10. From these figures, it is easy to see 
that, for Case II, OPL and ONMCNR have the same performances in accuracy for large $c0$, i.e., large $c0/K$, and similar performances in CPU time for very large $c0$. The reason for the latter is that, in this case, the computation complexities of $\|C^{0k}D_{0k}\|^2_F $ and $\|M^kN_k\|^2_F $ are similar. In addition, as before, ONC always performs quite well. For Case I, the four methods show the similar accuracy for different  $c0$.

\begin{figure}[H]
	\centering
	\includegraphics[width=14cm]{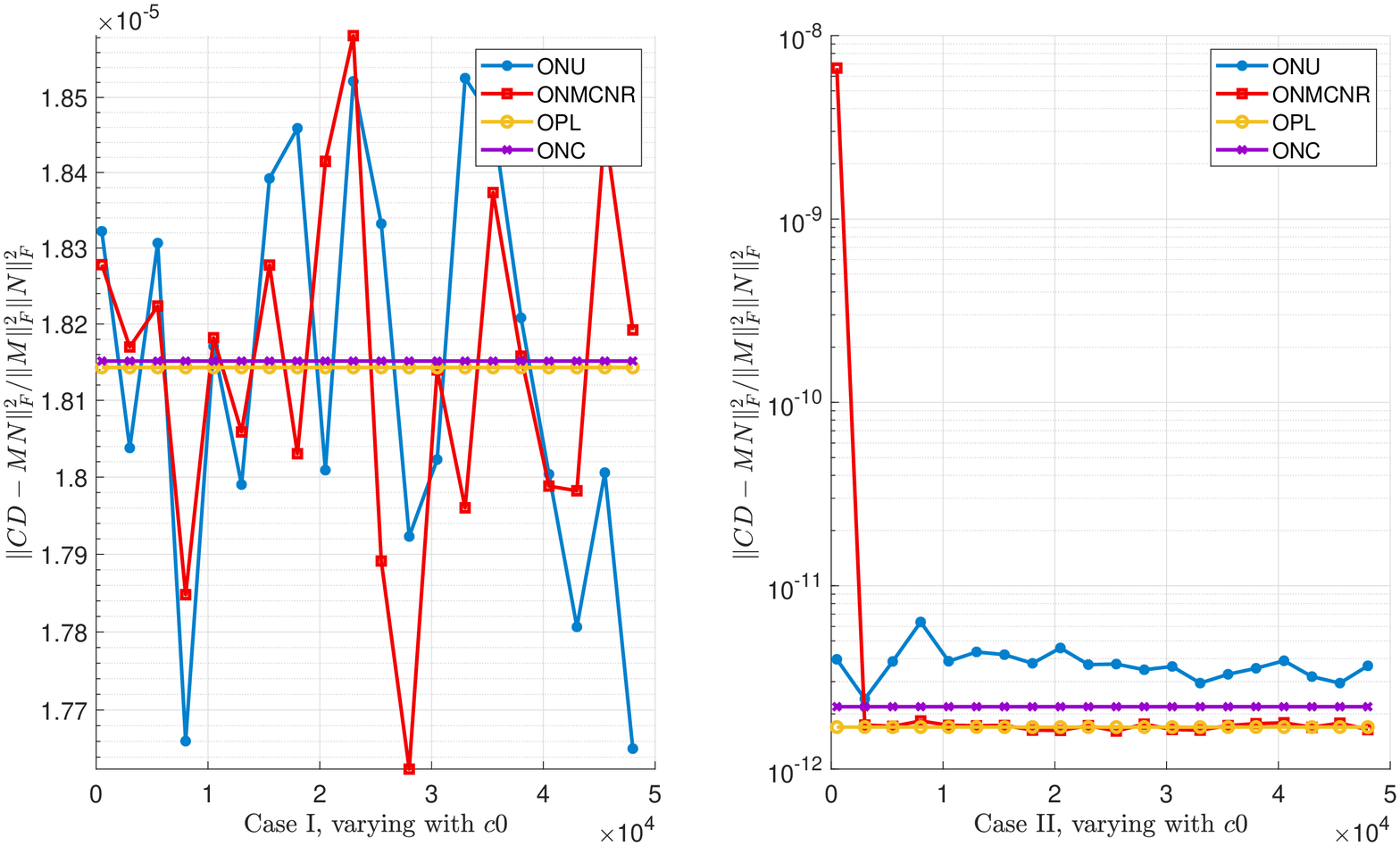} 
	\caption{Comparison of relative errors  for Algorithms 2 and  3  varying   with $c0$} 
	\label{fig59}
\end{figure}

\begin{figure}[H]
	\centering
	\includegraphics[width=14cm]{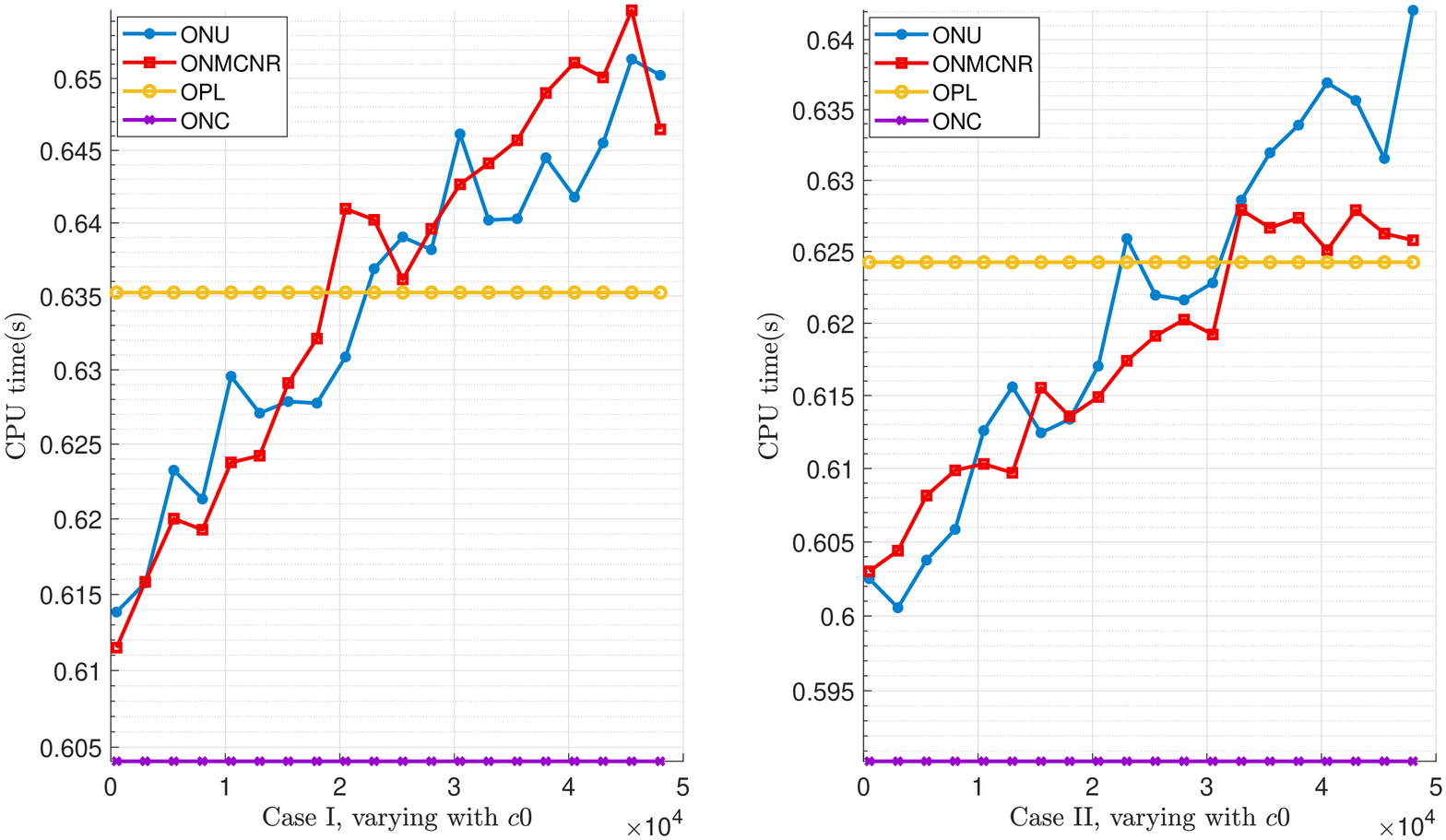} 
	\caption{Comparison of  CPU time for Algorithms 2 and  3  varying   with $c0$} 
	\label{fig510}
\end{figure}

In a word, for matrices whose rows or columns norms  are nonuniform, OPL performs best in accuracy in all cases and worst in CPU time in most cases. When $c0/K$ is large, ONMCNR and OPL have almost the same performances  in accuracy and CPU time. In addition, ONC always performs quite well.

\section{Concluding Remarks}
\label{sec.5}

We present the optimal sampling probabilities and sampling block sizes in the randomized sampling algorithm  for block matrix multiplication. Modified sampling block sizes and a two step algorithm for reducing the computation cost are also provided.   Numerical experiments show that our new methods outperform the SSM method in \cite{Charalambides2020} in accuracy with a little extra computation cost.

It is easy to see that the blocks of the matrices can be regarded as the single matrices scattered at multiple locations. So, the proposed methods are applicable to distributed data and  distributed computations and hence should have many potential important applications in the age of big data.

	\begin{appendices}
	
	\section{Proof of Theorem \ref{thm3.3}} \label{app.A}	
	
	We first  deduce that
	\begin{align*}
		\sum_{h = 1}^{m}\sum_{f = 1}^{p}\sigma^2
		&=\sum_{k=1}^{K}\sum_{i=1}^{n_k}\frac{\|M^{k(i)}\|^2_2\|N_{k(i)}\|^2_2}{c_kp_{k_i}}-\sum_{k=1}^{K}\frac{\|M^kN_k\|^2_F}{c_k}\\
		&\leq\frac{1}{\beta c }(\frac{\theta_2}{\theta_1})^\frac{1}{2}(\sum_{k=1}^{K}\sum_{i=1}^{n_k}\|M^{k(i)}\|_2\|N_{k(i)}\|_2)^2 -\frac{(1-\theta_2)}{c}(\frac{\theta_1}{\theta_2})^\frac{1}{2}(\sum_{k=1}^{K}\sum_{i=1}^{n_k}\|M^{k(i)}\|_2\|N_{k(i)}\|_2)^2\\
		&\leq\frac{\theta_2-\theta_1\beta+\theta_2\theta_1\beta}{\beta c(\theta_2\theta_1)^{1/2}}\|M\|^2_F\|N\|^2_F\\
		&=\frac{\varphi^2}{\beta c}\|M\|^2_F\|N\|^2_F,
	\end{align*}
	where the second inequality is derived from the Cauchy-Schwarz inequality. To prove \eqref{3.12}, we define a event $\theta$ as
	\begin{align*}
		{\|MN-{C}{D}\|_F }\leq{ \frac{\eta}{\sqrt{\beta} c}\|M\|_F\|N\|_F}.
	\end{align*}
	Thus,	as long as getting $\mathrm{Pr[\theta]}\geq1-\delta$,  \eqref{3.12} is proved. To explain easily, we define  a function
	\begin{align*}
		G(x)={\|MN-{C}{D}\|}^2_F
	\end{align*}
	with random variable $x=(1(i_1),\cdots,1(i_{c_1}),2(i_1),\cdots,2(i_{c_2}),\cdots,K(i_1),\cdots,K(i_{{c}_K}))$ standing for the positions of sampled results, where $k(i_t)$ denotes the  picked  $t$-th column (row) from the $k$-th block of $M$ ($N$), for  $k=1,\cdots,K$ and  $t=1,\cdots,c_k$.
	It is will be  shown that changing  one coordinate  $k(i_t)$ at a time  does not change  the value of $G$ too much.   Considering   $x$ and $x'$ differing only in the $k(i_t)$-th coordinate,  we can construct corresponding ${\|MN-{C}{D}\|}^2_F$ and ${\|MN-{C}^{'}{D}^{'}\|}^2_F$, respectively. Note that ${C}^{'}$ (${D}^{'}$)  differs from  $C$ ($D$) in only a single column(row). So, 
	\begin{align*}
		\|{C}{D}- {C}^{'}{D}^{'}\|_F&=\|\frac{M^{k(i_t)}N_{k(i_t)}}{{c}_kp_{k_{i_t}}}-\frac{M^{k(i_{t^{'}})}N_{k(i_{t^{'}})}}{{c}_kp_{k_{i_{t^{'}}}}}\|_F\\
		&\leq\frac{1}{{c}_kp_{k_{i_t}}}\|M^{k(i_t)}N_{k(i_t)}\|_F+\frac{1}{{c}_kp_{k_{i_{t^{'}}}}}\|M^{k(i_{t^{'}})}N_{k(i_{t^{'}})}\|_F\\
		&\leq\frac{2}{{c}_k}\frac{\|M^{k(r)}N_{k(r)}\|_F}{p_{k_{r}}}
		\leq\frac{2}{\beta {c}}(\frac{\theta_2}{\theta_1})^\frac{1}{2}\sum_{k = 1}^{K}\sum_{i=1}^{n_k}\|M^{k(i)}\|_2\|N_{k(i)}\|_2\\
		&\leq\frac{2}{\beta {c}}(\frac{\theta_2}{\theta_1})^\frac{1}{2}\|M\|_F\|N\|_F,
	\end{align*}
	where $\frac{\|M^{k(r)N_{k(r)}}\|_F}{p_{k_{r}}}=\mathop{\max }\limits_{i_t=1,\cdots,n_k }\frac{\|M^{k_{(i_t)}}\|_2\|N_{k_{(i_t)}}\|_2}{p_{k_{i_t}}}$.
	Furthermore, since
	\begin{align*}
		\|MN-{C}{D}\|_F&\leq\|MN-{C}^{'}{D}^{'}\|_F+\|{C}{D}- {C}^{'} {D}^{'}\|_F  \\
		&\leq\|MN-{C}^{'} {D}^{'}\|_F+\frac{2}{\beta {c}}(\frac{\theta_2}{\theta_1})^\frac{1}{2}\|M\|_F\|N\|_F\\
	\end{align*}
	and
	\begin{align*}
		\|MN-{C}^{'}{D}^{'}\|_F&\leq\|MN-{C}{D}\|_F+
		\|{C}{D}- {C}^{'}{D}^{'}\|_F\\
		&\leq\|MN-{C}{D}\|_F+\frac{2}{\beta {c}}(\frac{\theta_2}{\theta_1})^\frac{1}{2}\|M\|_F\|N\|_F,
	\end{align*}
	we have $|G(x)-G(x^{'})|\leq\|{C}{D}- {C}^{'}{D}^{'}\|_F$. For convenience,  let $\Delta$ denote $\frac{2}{\beta {c}}(\frac{\theta_2}{\theta_1})^\frac{1}{2}\|M\|_F\|N\|_F$ and $\gamma=\sqrt{2clog(1/{\delta})}\Delta$. Considering Hoeffding-Azuma inequality \cite{Mcdiarmid1989}, the probability inequality 
	\begin{align*}
		\mathrm{Pr}\mathit{[\|MN-{C}{D}\|_F\geq \frac{\varphi}{\sqrt{\beta c} }\|M\|_F\|N\|_F+
			\gamma]}\leq\rm{exp(-\frac{\gamma^2}{2c\Delta^2})}=\delta.
	\end{align*}
	is attained	and the theorem follows.
	
	\begin{remark}
	Let $\theta_2 = \theta_1<1$ and $\beta=1$,  we have $\eta=(\theta_2)^\frac{1}{2}+\sqrt{8log(1/\delta)}$  in Theorem \ref{thm3.3}. It is smaller than the one in \cite[Theorem 1]{Drineas2006}, i.e., $\eta=1+\sqrt{8log(1/\delta)}$.  This is because when computing the upper bound of 	$\sum_{h = 1}^{m}\sum_{f = 1}^{p}\sigma^2$,  we do not throw away the second item $-\sum_{k=1}^{K}\frac{\|M^kN_k\|^2_F}{c_k}$.
	\end{remark}

\end{appendices}

	
	\bibliography{MM_cite}
	
\end{document}